\documentclass[12pt]{amsart}
\usepackage{graphicx}
\usepackage[centertags]{amsmath}
\usepackage{amsfonts}
\usepackage{amssymb}
\usepackage{amsthm}

\newtheorem{thm}{Theorem}
\newtheorem*{thm*}{Theorem}
\newtheorem{lem}{Lemma}[section]

\newtheorem*{cor*}{Corollary}

\newtheorem{prop}[lem]{Proposition}
\theoremstyle{definition}

\theoremstyle{remark}
\newtheorem{rem}{Remark}[section]
\numberwithin{equation}{section}

\newcommand{\norm}[1]{\left\Vert#1\right\Vert}

\newcommand{\vphi}{{\varphi}}

\newcommand{\G}{\Gamma}

\newcommand{\DL}{\mathrm{D}_{\mathrm{L}}}
\newcommand{\dl}{\mathrm{d}_{\mathrm{L}}}

\newcommand{\calF}{\mathcal{F}}

\newcommand{\calO}{\mathcal{O}}

\newcommand{\bbZ}{\mathbb{Z}}

\newcommand{\bbQ}{\mathbb{Q}}
\newcommand{\bbR}{\mathbb R}
\newcommand{\bbC}{\mathbb C}

\newcommand{\bbN}{\mathbb N}

\newcommand{\fraka}{\mathfrak{a}}

\newcommand{\frakg}{\mathfrak{g}}
\newcommand{\frakk}{\mathfrak{k}}
\newcommand{\frakn}{\mathfrak{n}}

\newcommand{\m}{\mathrm{m}}

\newcommand{\SL}{ \mathrm{SL}}
\newcommand{\SO}{ \mathrm{SO}}

\newcommand{\PSL}{ \mathrm{PSL}}

\newcommand{\vol}{\mathrm{vol}}

\newcommand{\lap}{\triangle}
\newcommand{\bs}{\backslash}

\renewcommand{\Im}{\mathrm{Im}}

\newcommand{\Hd}{\mathbb{H}^{d}}
\newcommand{\Ld}{\mathbb{L}^{d}}

\begin{document}
\title[Poles and length spectrum]
{On distribution of poles of Eisenstein series and the length spectrum of hyperbolic manifolds}%
\author{Dubi Kelmer}%
\address{Department of Mathematics, 301 Carney Hall, Boston College
Chestnut Hill, MA 02467}
\email{dubi.kelmer@bc.edu}

\thanks{}%
\subjclass{}%
\keywords{}%

\date{\today}%
\dedicatory{}%
\commby{}%

\maketitle

\begin{abstract}
We extend results of Bhagwat and Rajan on a strong multiplicity one property for length spectrum to hyperbolic manifolds with cusps, showing that for two even dimensional hyperbolic manifolds of finite volume, if all but finitely many closed geodesics have the same length, then all closed geodesics have the same length. When the set of exceptional lengths is infinite, but sufficiently sparse, we can show that the two manifolds must have the same volume, and in low dimensions also the same number of cusps. A main ingredient in our proof is a generalization of a result of Selberg on the distribution of poles of Eisenstein series to hyperbolic manifolds.
\end{abstract}

\section*{Introduction}
The length spectrum of a hyperbolic manifold is the set of lengths of primitive closed geodesics listed with their multiplicities. It is an interesting question how much of the geometry of the manifold can be extracted from (partial) information on the length spectrum.
For compact hyperbolic surfaces a classical result of Huber \cite{Huber59}, using the trace formula, states that the length spectrum and the Laplace spectrum determine each other, as well as the area of the surface. This result was extended to noncompact finite area hyperbolic surfaces by M{\"u}ller \cite{Muller92}, where one needs to consider the Laplace spectrum together with the residual spectrum coming from poles of the Eisenstein series.

In higher dimensions the situation is more complicated since the geometric side of the trace formula depends on the complex length spectrum (i.e. lengths and holonomy) and not just lengths.  Nevertheless, using the analytic continuation of the Ruelle Zeta function, Bhagwat and Rajan \cite{BhagwatRajan11} showed that if two compact even dimensional hyperbolic manifolds have the same multiplicities for all but possibly finitely many exceptional lengths, then they must have the same length spectrum. In \cite{Kelmer12Spectra}, we refined their result and showed that one can allow an infinite, but sparse, set of possible exceptional lengths. Moreover, we showed that this data determines the Laplace spectrum and volume of the manifold (the question of whether the Laplace spectrum determines the length spectrum remains open). The main ingredient used in \cite{Kelmer12Spectra} was a more general version of the trace formula having the length spectrum appear directly on its geometric side obtained by combining different trace formulas corresponding to several representations of the Holonomy group.

The purpose of this note is to extend these results to hyperbolic manifolds with cusps. Instead of working with the trace formula directly as in \cite{Kelmer12Spectra}, we adopt the approach of \cite{BhagwatRajan11} and use the Ruelle Zeta function. Applying the results of Gon and Park on Selberg Zeta functions \cite{GonPark08,GonPark10} we can extend the result of \cite{BhagwatRajan11} to this setting. However, the refinement in \cite{Kelmer12Spectra} allowing an infinite exceptional set is more problematic. In this setting we only obtain a partial result showing that if the exceptional set is sufficiently sparse the two manifolds must have the same discrete Laplace spectrum and the same volume. In low dimensions we can also deduce that they have the same number of cusps. Even these partial results already require new results regarding the distribution of poles of Eisenstein series, which are of independent interest. To describe our results in more detail we need to introduce some notation.

Let $G\cong \SO_0(d,1)$ denote the group of isometries of hyperbolic $d$-space, $\Hd$.
Any finite volume hyperbolic manifold is of the form $X_\G=\G\bs \Hd$ where $\G<G$ is a torsion free lattice.
Given a hyperbolic manifold $X_\G$, for every $\ell\in (0,\infty)$ we denote by $\m_{\Gamma}(\ell)$ the number of primitive (i.e., wrapping once around) closed geodesics of length $\ell$ in $X_\Gamma$. For any two lattices $\Gamma_1,\Gamma_2< G$ let
\begin{equation}\label{e:Dls}\DL(\Gamma_1,\Gamma_2;T)=\sum_{\ell\leq T}|\m_{\Gamma_1}(\ell)-\m_{\Gamma_2}(\ell)|,\end{equation}
and
\begin{equation}\label{e:dls}\dl(\Gamma_1,\Gamma_2)=\limsup_{T\to\infty}\frac{\log(\DL(\Gamma_1,\Gamma_2;T))}{T}.\end{equation}
One can think of $\dl(\Gamma_1,\Gamma_2)$ as measuring the scaled density of the exceptional set of lengths having different multiplicities in the two manifolds, in particular, if this exceptional set is finite then $\dl(\Gamma_1,\Gamma_2)=0$. The result of \cite{Kelmer12Spectra} states that for two compact even dimensional hyperbolic manifolds, the condition that $d_L(\Gamma_1,\Gamma_2)<\frac{1}{2}$ already implies that the two manifolds have the same length spectrum. Moreover, it was shown there that (in any dimension) the weaker condition that $\mathrm{d}_L(\Gamma_1,\Gamma_2)<\frac{d-1}{2}$ implies that the two manifolds have the same Laplace spectrum (and hence, by Weyl's law, also the same volume).

Our first result extends \cite{BhagwatRajan11} to finite volume non-compact hyperbolic manifolds.
Since we rely on results of \cite{GonPark10} we need to impose a certain technical condition on the cusps of $X_\G$.
We say that $X_\G$ has neat cusps if for any parabolic subgroup $P<G$ with unipotent radical $N<P$ we have
$\G\cap P=\G\cap N$.
\begin{thm}\label{t:1}
Let $X_{\Gamma_1},X_{\Gamma_2}$ denote two even dimensional hyperbolic manifolds of finite volume with neat cusps. If $m_{\Gamma_1}(\ell)=m_{\Gamma_2}(\ell)$ for all $\ell\in \bbR$ except perhaps some finite exceptional set, then $m_{\Gamma_1}(\ell)=m_{\Gamma_2}(\ell)$ for all $\ell$.
\end{thm}

Next, we want to generalize the results of \cite{Kelmer12Spectra} allowing an infinite set of exceptions.
The relation between the length spectrum and discrete Laplace spectrum follows by more or less the same arguments as in the compact case. We note that when the manifold is not compact, the corresponding Weyl law includes both discrete and continuous spectrum, hence, the fact that the two manifolds have the same discrete spectrum no longer implies that they have the same volumes. Nevertheless, we show:
\begin{thm}\label{t:2}
Let $X_{\Gamma_1},X_{\Gamma_2}$ denote two $d$-dimensional hyperbolic manifolds of finite volume with neat cusps.
\begin{enumerate}
\item If $\dl(\Gamma_1,\Gamma_2)<\frac{d-1}{2}$ then the two manifolds have the same discrete Laplace spectrum and the same volume.
\item For $d=2,3$, if $\dl(\Gamma_1,\Gamma_2)<1/4$ then the two manifolds have the same number of cusps.
\end{enumerate}
\end{thm}

The proof of Theorem \ref{t:2} relies on results on the distribution of poles of Eisenstein series associated to hyperbolic manifolds, generalizing previous results of Selberg \cite{Selberg90} for hyperbolic surfaces. To describe these results we first recall some definitions and facts regarding these Eisenstein series.

Fix an Iwasawa decomposition, $G=NAK$, with $K$ maximal compact, $A$ Cartan, and $N$ unipotent, and let $P=NAM$ be a minimal parabolic where $M=Z_K(A)$ is the centralizer of $A$ in $K$. The cusps of $\Gamma$ are the $\Gamma$-conjugacy classes of minimal parabolic subgroups of $G$ intersecting $\Gamma$ nontrivially. Let $P_1,\ldots, P_\kappa$ denote a full set of representatives for these classes and let $k_i\in K$ such that $P_i=k_iPk_i^{-1}$. For each cusp $P_i=N_iA_iM_i$, let $\Gamma_{P_i}=\Gamma\cap P_i$ and $\Gamma_{N_i}=\Gamma\cap N_i$. The spherical Eisenstein series corresponding to the $i$'th cusp is the function on the upper half space
defined for $\Re(s)>d-1$ by the convergent series
\begin{equation}\label{e:EisensteinSph}E_i(s,z)=\sum_{\gamma\in \Gamma_{\!P_i}\bs \Gamma} y_i(\gamma.z)^s,
\end{equation}
where we use the coordinates $z=(x,y)\in \bbR^{d-1}\times \bbR^+$ for the upper half space and set $y_i(z)=y(k_i^{-1}.z)$. The constant term of $E_i(s,z)$ with respect to the $j$'th cusp is defined by
\begin{equation}\label{e:ConstSph} E_{ij}(s,z)=\frac{1}{v_j}\int_{\Gamma_{N_j}\bs N_j}E_i(s,n.z)dn,\end{equation}
where $v_j=\vol(\Gamma_{N_j}\bs N_j)$, and $dn$ is Haar measure on $N_j$. These constant terms satisfy
\begin{equation}\label{e:ConstSph1}
E_{ij}(s,z)=\delta_{ij}y_j(z)^s+\phi_{ij}(s)y_j(z)^{d-1-s},
\end{equation}
with $\phi_{ij}(s)$  the coefficients of the scattering matrix.

The Eisenstein series $E_j(s,z)$, the scattering matrix $\phi(s)=(\phi_{ij}(s))$, and its determinant $\vphi(s)=\det(\phi(s))$ (a priori defined for $\Re(s)>d-1$) have a meromorphic extension to the complex plane and satisfy the functional equation $\phi(s)\phi(d-1-s)=I$. The poles of $\vphi(s)$ which are also the poles of the Eisenstein series, are all in the half plane $\Re(s)<\tfrac{d-1}{2}$ except for at most finitely many poles in the interval $(\tfrac{d-1}{2},d-1]$. From the functional equation $\vphi(s)\vphi(d-1-s)=1$, we can understand the distribution of these poles by looking at the zeroes of $\vphi(s)$ in the half plane $\Re(s)>\tfrac{d-1}{2}$. For these we show the following.
\begin{thm}\label{t:ZeroDist1}
The zeroes of the scattering determinant, $\vphi(s)$, in the half plane $\Re(s)>\tfrac{d-1}{2}$ are all located in some vertical strip, and writing these zeroes as $\rho=\beta+i\gamma$ (with multiplicities) we have:
\begin{enumerate}
\item There is a constant $A_\G$ such that
\begin{equation}\label{e:ZeroDist1}
\mathop{\sum_{|\gamma|<T}}_{\beta>\tfrac{d-1}{2}}(\beta-\tfrac{d-1}{2})=\frac{\kappa(d-1)}{2\pi} T\log(T)+A_\Gamma T+ O(\log(T))
\end{equation}
\item
For any $\alpha\geq \alpha_0=d-\tfrac{5}{4}$
\begin{equation}\label{e:ZeroDist2}
\mathop{\sum_{|\gamma|<T}}_{\beta>\alpha}(\beta-\alpha)\ll T\min\{\log(\tfrac{1}{(\alpha-\alpha_0)}),\log\log T\}
\end{equation}
\end{enumerate}
\end{thm}


\begin{rem}
Using the relation between the zeroes and poles one can interpret this result as saying that a hundred percent of the poles $\tilde{\rho}=\tilde{\beta}+i\gamma$ of $\vphi(s)$ in the half plane $\Re(s)<\tfrac{d-1}{2}$ are concentrated in the strip $\tfrac{1}{4}\leq \Re(s)<\tfrac{d-1}{2}$, in the sense that, for any $\alpha<\tfrac{1}{4}$,  as $T\to\infty$
$$\mathop{\sum_{|\gamma|<T}}_{\alpha<\tilde{\beta} <\tfrac{d-1}{2}}\!\!\!(\tfrac{d-1}{2}-\tilde{\beta})\sim \mathop{\sum_{|\gamma|<T}}_{\tilde{\beta} <\tfrac{d-1}{2}}(\tfrac{d-1}{2}-\tilde{\beta}).$$
\end{rem}

\begin{rem} For hyperbolic surfaces this result is due to Selberg \cite{Selberg90} and the value of $\alpha_0=\tfrac{3}{4}$ is best possible. Indeed, for  $\Gamma=\SL_2(\bbZ)$ the scattering determinant can be computed explicitly in terms of the Riemann Zeta function and its poles are located at the zeroes of $\zeta(1-2s)$, hence, a positive proportion\footnote{The Riemann hypothesis implies that all the zeroes are on that line.} are on the line $\Re(s)=\tfrac{1}{4}$. For $3$-manifolds, when $\Gamma=\SL_2(\calO_K)$ with $\calO_K$ the ring of integers of a quadratic complex number field, the poles of the scattering determinant are at the zeros of the Dedekind Zeta function $\zeta_K(1-s)$ and \eqref{e:ZeroDist2} holds with $\alpha_0=\tfrac{3}{2}$. However, for general $\Gamma< \PSL_2(\bbC)$ our method only gives a weaker result with $\alpha_0=\tfrac{7}{4}$. 
\end{rem}


\section{Zeta functions of Selberg and Ruelle}
\subsection{Selberg Zeta functions}
We recall the correspondence between closed geodesics on $X_\G$ and hyperbolic conjugacy classes in $\G$.
Let $\frakg=\frakn\oplus \fraka\oplus \frakk$ denote the decomposition of the Lie algebra $\frakg$ corresponding to the Iwasawa decomposition $G=NAK$.
Fix once and for all an element $H_0\in \fraka$ such that $\rho(H_0)=\frac{d-1}{2}$, where $\rho$ denotes half the sum of positive roots, and let $a_t=\exp(tH_0)\in A$. Any hyperbolic $\gamma\in \Gamma$ is conjugated in $G$ to an element $m_\gamma a_{\ell_\gamma}\in MA^+$ where $A^+=\{a_t|t>0\}$ and $M$ is the centralizer of $A$ in $K$. The pair $(\ell_\gamma,m_{\gamma})$ is the length and holonomy class of the closed geodesic corresponding to $\gamma$, where $\ell_\gamma$ is uniquely determined by the conjugacy class of $\gamma$ and $m_\gamma$ is determined up to conjugacy in $M$.

The Selberg Zeta function, $Z_\G(\sigma,s)$, corresponding to an irreducible representation $\sigma\in \hat{M}$ is defined on the half plane $\Re(s)>d-1$ by
\begin{eqnarray}\label{e:SZF}\nonumber Z_\G(\sigma,s)
 =\exp\left(-\sum_{\gamma\in \Gamma_h'}\sum_{j=1}^\infty\frac{\overline{\chi_{\sigma}(m_{\gamma^j})}}{j D(\gamma^j)}e^{-(s-\frac{d-1}{2})j\ell_\gamma}\right)
\end{eqnarray}
where $\G_h'$ denotes the set of primitive hyperbolic conjugacy classes 
and
\[D(\gamma)=e^{\frac{d-1}{2}\ell_\gamma}|\det\big(\mathrm{Ad}(m_\gamma a_\gamma)^{-1}-\mathrm{Id})|_\frakn\big)|.\]

Before we present the result of \cite{GonPark10} on the analytic continuation of these Zeta functions we need to recall some more background on Eisenstein series. In addition to the spherical Eisenstein series given in \eqref{e:EisensteinSph}, for any representation class $\sigma\in \hat{M}$ one can define corresponding Eisenstein series, scattering matrix, and scattering determinant. We refer the reader to \cite{Warner79,GonPark10} for the precise definition, and just note here the following general result from \cite[Section 6]{Muller89} regarding the poles of these scattering determinants.
\begin{prop}\label{p:Muller}
Let $\vphi_{\Gamma,\sigma}(s)$ denote the scattering determinant corresponding to $\sigma\in \hat{M}$. Then $\vphi_{\G,\sigma}(s)$ is a meromorphic function with all of its poles in the half plane $\Re(s)<\tfrac{d-1}{2}$ except for finitely many poles in the interval $(\tfrac{d-1}{2},d-1]$. Moreover if we denote by $S_{\G,\sigma}$ the set of poles in $\Re(s)<\tfrac{d-1}{2}$ then
$\sum_{\eta\in S_{\G,\sigma}}\frac{\Re(\eta-\tfrac{d-1}{2})}{|\eta-\tfrac{d-1}{2}|^2}$ converges.
\end{prop}
We can now state the result of Gon and Park \cite[Theorem 4.6]{GonPark10}.
For $k=0,\ldots, [\frac{d-1}{2}]$, let $\sigma_k\in \hat{M}$ correspond to the irreducible representation of $M$ on the space $\bigwedge^k(\bbC^{d-1})$ (when $d=2n+1$ and $k=n$ we denote by $\sigma_n^{\pm}$ the two unramified irreducible representations acting on $\bigwedge^n(\bbC^{2n})$. We also denote by
$\lap_k$ the Laplacian acting on $k$-forms (where $\lap_0=\lap$ is just the hyperbolic Laplacian).
\begin{prop}\label{p:GonPark}
The Selberg Zeta function $Z_\G(\sigma_k,s)$ has a meromorphic continuation to the complex plane. It has poles at the points $s=\frac{d-1}{2}-\ell$, $\ell\in \bbN\cup \{0\}$ (and additional zeroes/poles at negative integers when $d$ is even). It has spectral zeroes at the points $s=\tfrac{d-1}{2}\pm ir$ with $\lambda=r^2+(\tfrac{d-1}{2}-k)^2$ an eigenfunction of $\lap_k$, and residual zeroes at $\eta\in S_{\Gamma,\sigma_k}$. It also has finitely many residual poles in the interval $(\tfrac{d-1}{2},d-1]$ at the poles of $\vphi_{\sigma_k}$. The order of the spectral zeroes equals the multiplicity of the corresponding eigenvalue and the order of the residual poles and zeroes are equals to $\dim(\sigma_k)$ times the order of the corresponding pole of $\vphi_{\sigma_k}$.
\end{prop}

\subsection{Ruelle Zeta function}
Information about the length spectrum is captured directly by the Ruelle Zeta function, defined in the half plane $\Re(s)>\frac{d-1}{2}$ by the Euler product
\[R_{\Gamma}(s)=\prod_{\gamma\in \Gamma_h'}(1-e^{-s\ell_\gamma}).\]
As shown in \cite{GonPark10}, when $d$ is even it is related to the Selberg Zeta functions  via
\begin{equation}\label{e:RF1} R_\Gamma(s)=\prod_{k=0}^{d/2-1}\left[\frac{Z_\Gamma(\sigma_k,s+k)}{Z_\Gamma(\sigma_k,s+d-1-k)}\right]^{(-1)^{k+1}},\end{equation}
and when $d$ is odd we have
\begin{equation}\label{e:RF2}R_\Gamma(s)=\prod_{k=0}^{\tfrac{d-1}{2}}\left[ Z_\Gamma(\sigma_k,s+k) Z_\Gamma(\sigma_k,s+d-1-k)\right]^{(-1)^{k+1}},\end{equation}
where in the odd case $d=2n+1$ we denoted by $Z_\Gamma(\sigma_n,s):=Z_\Gamma(\sigma_n^+,s)Z_\Gamma(\sigma_n^-,s)$.

We also consider a variant of the Ruelle Zeta function that is similar to the Selberg Zeta function for surfaces, that is,
\begin{equation}\label{e:Zeta}Z_\Gamma(s)=\prod_{\gamma\in \Gamma_h'}\prod_{a=0}^\infty(1-e^{-(s+a)\ell_\gamma}).\end{equation}
When $d$ is even, a simple manipulation of \eqref{e:RF1} shows that this Zeta function can be expressed as a finite product of Selberg Zeta functions and their inverses as
\begin{equation}\label{e:Zeta2}
Z_\Gamma(s)=\prod_{k=0}^{d/2-1}\left[\prod_{j=k}^{d-2-k}Z_\Gamma(\sigma_k,s+j)\right]^{(-1)^{k+1}}.
\end{equation}

\section{Proof of Theorem 1}
Let $\Gamma_1,\Gamma_2\subset \SO(d,1)$ be two torsion free lattices with $d$ even. Assume that $m_{\Gamma_1}(\ell)=m_{\Gamma_2}(\ell)$ for all $\ell\in \bbR$ except perhaps for a finite exceptional set $\{\ell_1,\ldots, \ell_N\}$.  For these exceptional lengths let $\Delta m(\ell_j)=m_{\Gamma_1}(\ell_j)-m_{\Gamma_2}(\ell_j)$. We will show that $\Delta m(\ell_j)=0$ as well.

To do this consider the quotient $F(s)=\frac{Z_{\Gamma_1}(s)}{Z_{\Gamma_2}(s)}$ of the corresponding Zeta functions given in \eqref{e:Zeta} and note that for $\Re(s)>d-1$ we have
\[F(s)=\prod_{j=1}^N \prod_{a=0}^\infty(1-e^{-(s+a)\ell_j})^{\Delta m(\ell_j)}.\]
This product absolutely and uniformly converges on any compact set away from
$\{s\in \bbC| e^{-(s+a)\ell_j}= 1,\; j=1,\ldots, N,\; a\in \bbN\cup \{0\}\}$. Consequently, $F(s)$ is a meromorphic functions with all its zeros and poles located at the points $\rho_{a,b,j}=-a+\frac{2\pi i b}{\ell_j}$ with $a,b\in \bbZ$, $a\geq 0$. The order of the pole/zero $\rho_{a,b,j}$ does not depend on $a$ and is given by
$$\sum_{\{i|\ell_i b\in \ell_j\bbZ\}}\Delta m(\ell_i).$$
In order to avoid any possible cancelation in this sum, let $\ell_1$ be the largest for which $\Delta m(\ell_1)\neq 0$. If there are other $\ell_i$ with $\frac{\ell_i}{\ell_1}\in\bbQ$ and $\Delta m(\ell_i)\neq 0$ let $q\in \bbN$ denote their least common multiple. Now for all $b\in \bbN$ with $(b,q)=1$ we have that $\ell_i b/\ell_1\in \bbZ$ if and only if $\ell_i=\ell_1$, and hence $F(s)$ has poles/zeroes at all points $\rho_{a,b}=-a+\frac{2\pi i b}{\ell_1}$ with $a\geq 0$ and  $(b,q)=1$, with the same order $\Delta m(\ell_1)\neq 0$.

Next use \eqref{e:Zeta2} to express $F(s)$ as a finite product of quotients of Selberg Zeta functions.
Since $Z_{\Gamma_i}(\sigma_k,s+j)$ have no complex poles in the half plane $\Re(s)<-j$, all of the poles/zeroes of $F(s)$ come from the residual zeroes of $Z_{\Gamma_i}(\sigma_k,s+j)$ with $0\leq k\leq d/2-1$, and $k\leq j\leq d-2-k,\; i=1,2$. This means that for any pair $a,b\in \bbN$ with $a\geq d-1$ and $(b,q)=1$, there is some $j_{a,b}\leq d-1$ such that $\rho_{a-j_{a,b},b}\in S_{\Gamma_i,\sigma_k}$ is a residual zero of $Z_{\Gamma_i}(\sigma_k,s)$, for some $0\leq k\leq d/2-1$ and $i\in\{1,2\}$. We thus get a bound
\[\sum_{a=d}^\infty \mathop{\sum_{b=1}^\infty}_{(b,q)=1}\frac{a-j_{a,b}+\frac{d-1}{2}}{(a-j_{a,b}+\frac{d-1}{2})^2+4\pi^2 b^2/\ell_1^2}\leq \sum_{i=1}^2\sum_{k=0}^{d/2-1}\sum_{\eta\in S_{\Gamma_i,\sigma_k}}\frac{|\Re(\eta-\frac{d-1}{2})|}{|\eta-\frac{d-1}{2}|^2},\]
noting that all terms are positive and every summand on the left also appears on the right.
However, the sum on the right converges by Proposition \ref{p:Muller} while the sum on the left clearly diverges, implying that $\Delta m(\ell_1)= 0$ as claimed.

\section{Proof of Theorem 2}
To prove Theorem 2 consider the quotient of the Ruelle Zeta functions
$$F(s)=\frac{R_{\Gamma_1}(s)}{R_{\Gamma_2}(s)}=\prod_{\ell } (1-e^{-s\ell})^{\Delta m(\ell)}.$$
A priori, this equality holds for $\Re(s)>d-1$, however, if we know that
\begin{equation}\label{e:expbound}
\sum_{\ell\leq T}|\Delta m(\ell)|=O(e^{cT}),
\end{equation}
then the right hand side absolutely converges for $\Re(s)>c$ and hence defines an analytic function that has no zeroes or poles on that half plane.

The condition $\dl(\Gamma_1,\Gamma_2)<\tfrac{d-1}{2}$ implies that \eqref{e:expbound} holds with some $c<\tfrac{d-1}{2}$ and hence the zeroes and poles of $R_{\Gamma_1}(s)$ and $R_{\Gamma_2}(s)$ in the half pane $\Re(s)>c$ must be the same. From the factorization of $R_{\Gamma_i}(s)$ as a product of Selberg Zeta functions in \eqref{e:RF1} \eqref{e:RF2} together with the location of the poles and zeroes of the Zeta functions $Z_\Gamma(\sigma_k,s)$ given in Proposition \ref{p:GonPark}, we see that $R_{\Gamma_i}(s)$ has no zeroes in $\Re(s)\geq \frac{d-1}{2}$ and its poles there all come from the spectral zeroes of $Z_\G(\sigma_0,s)$ with $\sigma_0$ the trivial representation. That is, $R_{\Gamma_i}(s)$ has poles at $s=\frac{d-1}{2}+ir$ whenever $\lambda=\frac{(d-1)^2}{4}+r^2$ is an eigenvalue of $\lap$ with eigenfunction in $L^2(X_{\Gamma_i})$ with order given by the multiplicity of the eigenvalue. In particular, this shows that the discrete spectrum of $\Gamma_1$ and $\Gamma_2$ must be the same.

To see that the volumes are also equal we use the Weyl law, which for non-compact hyperbolic manifolds takes the form,
$$N_\G(T)-\frac{1}{2\pi}\int_{-T}^T \frac{\vphi_{\G}'}{\vphi_\G}(\tfrac{d-1}{2}+it)dt=C_d\vol(X_\G)T^d+O(T^{d-1})+O(T\log(T)),$$
where $N_{\Gamma}(T)$ denotes the number of Laplace eigenvalues $\lambda< \tfrac{(d-1)^2}{4}+T^2$ and $C_d$ is a constant depending only on $d$.
The term involving the scattering determinant can be further evaluated by counting poles. Specifically if we let
$$S_\G(T)=\{\eta\in S_{\Gamma,\sigma_0}:|\Im(\eta|\leq T\},$$
following the same argument as in \cite[equation (0.15)]{Selberg90}, we see that
$$|S_\G(T)|=-\frac{1}{2\pi}\int_{-T}^T \frac{\vphi_{\G}'}{\vphi_\G}(\tfrac{d-1}{2}+it)dt+O(T),$$
 and the Weyl law takes the form
\begin{equation}\label{e:Weyl}
N_{\Gamma}(T)+|S_\G(T)|=C_d\vol(X_{\Gamma})T^d+O(T^{d-1})+O(T\log(T)).
\end{equation}
Comparing this for the two lattices, noting that $N_{\Gamma_1}(T)=N_{\Gamma_2}(T)$ and bounding all error terms by, say $O(T^{d-1/2})$, we see that
$$|S_{\G_1}(T)|-|S_{\G_2}(T)|=C_d(\vol(X_{\Gamma_1})-\vol(X_{\Gamma_2}))T^d+O(T^{d-1/2}).$$

Now let $c'=\max\{c,\tfrac{d-3}{2}\}$ so that $\frac{R_{\Gamma_1}(s)}{R_{\Gamma_2}(s)}$ has no poles or zeros in $c'<\Re(s)<\tfrac{d-1}{2}$. In this range all zeroes of $R_{\Gamma_i}(s)$ are the residual zeros of $Z_{\Gamma_i}(\sigma_0,s)$, and hence the residual zeroes of $Z_{\Gamma_1}(\sigma_0,s)$, and $Z_{\Gamma_2}(\sigma_0,s)$ in $\Re(s)>c'$ must be the same. Fix some $c'<c_1<\tfrac{d-1}{2}$ and let
$$S_\G(c_1,T)=\{\eta\in S_\G(T): \Re(\eta)\leq c_1\}.$$
Recalling the relation between the residual zeroes and the poles of the scattering matrix we get that
$$|S_{\G_1}(T)|-|S_{\G_2}(T)|=|S_{\G_1}(c_1,T)|-|S_{\G_2}(c_1,T)|,$$
where all sets are counted with multiplicities.
Finally, recalling that $\eta\in S_\G$ if and only if $d-1-\eta$ is a zero of $\vphi_\G$ and using \eqref{e:ZeroDist1} we get that for each lattice
\begin{eqnarray*}
|S_{\G_i}(c_1,T)|&\leq&\frac{1}{\tfrac{d-1}{2}-c_1}\sum_{\eta \in S_{\G_i}(c_1,T)}(\tfrac{d-1}{2}-\Re \eta)\\
&\ll& \sum_{\eta \in S_{\G_i}(T)}(\tfrac{d-1}{2}-\Re\eta)
\ll T\log(T)
\end{eqnarray*}
We thus get that $||S_{\G_1}(T)|-|S_{\G_2}(T)||\ll T\log T$ and hence
\begin{eqnarray*}
C_d(|\vol(X_{\Gamma_1})-\vol(X_{\Gamma_2})|)T^d&\ll& T^{d-1/2},
\end{eqnarray*}
implying that $\vol(X_{\Gamma_1})=\vol(X_{\Gamma_2})$.

For the results on the number of cusps, let $d\leq 3$ and assume that $\dl(\G_1,\G_2)<\tfrac{1}{4}$, so that $D_L(\Gamma_1,\Gamma_2,T)=O(e^{cT})$ with some $c<\tfrac{1}{4}$. Comparing the Zeta functions as before we see that the poles of $\vphi_{\Gamma_1}$ and $\vphi_{\Gamma_2}$ in the half plane $\Re(s)>c$ are the same.
Hence
\begin{eqnarray*}\lefteqn{\sum_{\eta\in S_{\Gamma_1}(T)}(\tfrac{d-1}{2}-\Re\eta)-\sum_{\eta\in S_{\Gamma_2}(T)}(\tfrac{d-1}{2}-\Re\eta)}\\
&&=\sum_{\eta\in S_{\Gamma_1}(c,T)}(\tfrac{d-1}{2}-\Re\eta)-\sum_{\eta\in S_{\Gamma_2}(c,T)}(\tfrac{d-1}{2}-\Re\eta).
\end{eqnarray*}
By \eqref{e:ZeroDist1}, the left hand side equals $\frac{(\kappa_1-\kappa_2)(d-1)}{2\pi}T\log(T)+O(T)$, hence, if we can bound each one of the sums on the right by
$O(T\log\log(T))$ we would get that $\kappa_1=\kappa_2$ as claimed.

To bound the sum on the right, for $\Gamma=\G_i$ any one of the two lattices we can write each pole as $\eta=d-1-\beta-i\gamma$ with $\beta+i\gamma$ a zero of $\vphi_\G(s)$. Setting $\alpha=d-1-c$ we get
\begin{eqnarray*}
\sum_{\eta\in S_{\Gamma}(c,T)}(\tfrac{d-1}{2}-\Re\eta)=\mathop{\sum_{|\gamma|<T}}_{\beta \geq \alpha}(\beta-\tfrac{d-1}{2}).
\end{eqnarray*}
The condition $c<\tfrac{1}{4}$ implies $\alpha>\alpha_0$ and we can bound
\begin{eqnarray*}
\mathop{\sum_{|\gamma|<T}}_{\beta \geq \alpha_0}(\beta-\alpha_0)&\geq& \mathop{\sum_{|\gamma|<T}}_{\beta \geq \alpha}(\beta-\alpha_0)
\geq(\alpha-\alpha_0)\mathop{\sum_{|\gamma|<T}}_{\beta \geq \alpha}1\\
&= & \left(\tfrac{\alpha_0-\alpha}{\alpha-\frac{d-1}{2}}\right)\mathop{\sum_{|\gamma|<T}}_{\beta \geq \alpha}(\alpha-\beta+\beta-\tfrac{d-1}{2}),
\end{eqnarray*}
hence
\begin{eqnarray*}
\mathop{\sum_{|\gamma|<T}}_{\beta \geq \alpha}(\beta-\tfrac{d-1}{2})\leq \left(\tfrac{\alpha-\frac{d-1}{2}}{\alpha_0-\alpha}\right)\mathop{\sum_{|\gamma|<T}}_{\beta \geq \alpha_0}(\beta-\alpha_0)+\mathop{\sum_{|\gamma|<T}}_{\beta \geq \alpha}(\beta-\alpha),
\end{eqnarray*}
and from \eqref{e:ZeroDist2} the sums on the right are bounded by $O(T\log\log(T))$ concluding the proof.

\begin{rem}
 If we assume that $d_L(\Gamma_1,\Gamma_2)<\tfrac{1}{2}$ we get that all the poles and zeroes of $R_{\Gamma_1}(s)$ and $R_{\Gamma_2}(s)$ cancel out in the half plane $\Re(s)\geq \tfrac{1}{2}$. However, in contrast to the compact case, we cannot deduce from this that Zeta functions are the same because we cannot exclude the possibility that the spectral zeroes of $Z_{\Gamma_i}(\sigma_k,s)$ will cancel out with residual poles of $Z_{\Gamma_i}(\sigma_{k'},s_j)$ for some $k'<k$.
\end{rem}

\begin{rem}
The result on the number of cusps is restricted to dimensions $d=2,3$ for a similar reason. In general, for $d>3$ our result shows that almost all poles of $Z_\G(\sigma_0,s)$ are in the half plane $\Re(s)\geq \frac{1}{4}$, but even if we assume that $\dl(\Gamma_1,\Gamma_2)<\frac{1}{4}$ we can't conclude that the number of cusps is the same because of possible cancelation with zeroes of the other Zeta functions. 
\end{rem}

%

\section{Distribution of Poles of Eisenstein series}\label{s:DistPoles}
 We conclude with the proof of Theorem \ref{t:ZeroDist1}, generalizing the results of \cite{Selberg90} on the distribution of poles of Eisenstein series to hyperbolic manifolds of higher dimensions. Our proof is very similar to Selberg's original proof, that is, we express the scattering coefficients as certain Dirichlet series with positive coefficients and then use general results about such Dirichlet series to study the distribution of the zeroes of the scattering determinant. The only new difficulty is coming from the fact that in higher dimensions these Dirichlet series have a faster growth rate on the critical line, requiring some modifications of the arguments in \cite{Selberg90}.

\subsection{Dirichlet series with positive coefficients}
We prove some results about general Dirichlet series with positive coefficients of the form,
\begin{equation}\label{e:fseries}
f(s)=\sum_{n=1}^\infty \frac{a_n}{\lambda_n^s},
\end{equation}
with all $\lambda_n,a_n>0$. In particular, we need the following generalization of Selberg's \cite[Lemma 3]{Selberg90}
\begin{prop}\label{l:boundintf}
Let $f(s)$ be as in \eqref{e:fseries}. Assume that the series absolutely converges for $\Re(s)>1$, that it has an analytic continuation to $\Re(s)>0$ with a simple pole at $s=1$ and at most finitely many poles in the strip $0<\Re(s)<1$, all in the interval (0,1). Further assume that $f(s)$ has a continuous extension to $\Re(s)=0$ and it satisfies the growth condition
\begin{equation}\label{e:tbound}
|f(\sigma+it)|=O(|t|^{r}),\quad 0\leq \sigma <3/2,\; t\gg 1.
\end{equation}
for some $r\geq 1$. Let $\sigma_1=\frac{4r-1}{4r}$, then for all $\sigma\geq \sigma_1$ we have
\[\frac{1}{T}\int_1^T|f(\sigma+it)|^2dt\ll \min(\frac{1}{(\sigma-\sigma_1)^2},\log^2(T)).\]
\end{prop}
\begin{rem}
For $r< 1$ our proof gives a similar result with $\sigma_1=\tfrac{3}{4}$ and an error of $\log(T)$ instead of $\log^2(T)$.
This is worse than the original result of Selberg's \cite[Lemma 3]{Selberg90} who showed this for $r=1/2$ with a better value of $\sigma_1=\tfrac{1}{2}$.
Our proof goes along the same line as his, except in one point where his argument seems to work only when $r\leq 1/2$ and we had to make some modifications, resulting in a slightly worse value for $\sigma_1$.
\end{rem}

Before we proceed with the proof we recall a couple of standard results on these Dirichlet series.
\begin{lem}\label{l:Ax}
Let $f(s)$ be as in Proposition  \ref{l:boundintf} and let
\begin{equation}\label{e:Ax}
A_f(x)=\sum_{\lambda_n\leq x} a_n.
\end{equation}
Then
\begin{equation}\label{e:Af}
A_f(x)=a x+\sum_{j=1}^k x^{\rho_j}p_j(\log(x))+O(x^{1-\frac{1}{2r}}\log(x)),
\end{equation}
where $a=\mathrm{Res}_{s=1}f(s)$, $\rho_1,\ldots,\rho_k$ are the poles of $f(s)$ in $(0,1)$, and $p_j$ are certain polynomials.
\end{lem}
\begin{proof}
This follows from standard techniques of analytic number theory. Specifically,
let
$$B_f(x)=\sum_{\lambda_n\leq x}a_n(1-\tfrac{\lambda_n}{x}),$$
that is, $B_f(x)=\frac{1}{x}\int_1^x A_f(t)dt$. For any $c>1$ we can write $B_f(x)=\frac{1}{2\pi i}\int_{\Re(s)=c}\frac{f(s)x^s}{s(s+1)}ds$. Taking $c=1+\frac{1}{\log(x)}$  and noting that $|f(s)|\ll \frac{1}{c-1}\ll \log(x)$ is uniformly bounded there we get that
$$B_f(x)=\frac{1}{2\pi i}\int_{c-iT}^{c+iT}\frac{f(s)x^s}{s(s+1)}ds+O(\frac{x\log(x)}{T}).$$
Shifting the contour of integration to $\Re(s)=\epsilon$ (and taking $\epsilon\to 0$) we get that
\begin{eqnarray*}\frac{1}{2\pi i}\int_{c-iT}^{c+iT}\frac{f(s)x^s}{s(s+1)}ds &=& \frac{ax}{2}+\sum_j x^{\rho_j}\tilde{p}(\log(x))\\
&&+O(T^{r-1}\log(T))+O(\frac{x\log(T)}{T^2\log(x/T^r)})
\end{eqnarray*}
where the explicit terms come from the residues at the poles, we used $f(\epsilon+it)\ll t^r$ to bound the integral on the vertical line $\Re(s)=\epsilon$ by $O(T^{r-1}\log(T))$, and the convexity bound $f(\sigma+iT)\ll T^{r(1-\sigma)}\log(T)$ to bound the horizontal integral from $\epsilon+iT$ to $c+iT$ by $O(\frac{x\log(T)}{T^2\log(x/T^r)})$. (For $r<1$ the first error term would be $O(1)$).
Putting it all together with a choice of $T=\left(\frac{x}{2}\right)^{1/r}$ gives
\begin{equation}\label{e:Bx}
xB_f(x)=F(x)+O(x^{2-\frac{1}{r}}\log(x)),
\end{equation}
where $F(x)=\frac{ax^2}{2}+\sum_j x^{1+\rho_j}\tilde{p}_j(\log(x)$. Finally, since
$\int_{1}^xA_f(t)dt=xB_f(x)$ and $A_f(x)$ is increasing we get that \eqref{e:Bx} implies
$$A_f(x)=F'(x)+O(x^{1-\frac{1}{2r}}\log(x)),$$
concluding the proof.
\end{proof}

\begin{lem}\label{l:ff*}
Let $f(s)$ be as above. For a large parameter $x\geq1$ and $k\in \bbN$ let $f^*(s)=f^*_{x,k}(s)$ be defined by the finite series
\begin{eqnarray}\label{e:f*}
f^*(s)&=&\sum_{\lambda_n\leq x}a_n(1-\frac{\lambda_n}{x})^k\lambda_n^{-s}
\end{eqnarray}
Let $0<\sigma_0<1$ and $k>r$. Then for any $\sigma_0\leq \sigma\leq 3/2$ and $x^{\frac{1-\sigma_0}{k+1}}\leq t\leq x^{\frac{\sigma_0}{r}}$ we have
\begin{equation}\label{e:f*f}f_{x,k}^*(s)=f(s)+O(1).\end{equation}
\end{lem}
\begin{proof}
For any $c>1$ we can write
$$f^*(s)=\frac{k!}{2\pi i}\int_{c-i\infty}^{c+i\infty} \frac{x^{z-s}f(z)}{(z-s)(z-s+1)\cdots(z-s+k)}dz.$$
Shifting the contour of integration to $\Re(z)=\epsilon$ (and taking $\epsilon\to 0$) the pole at $z=s$ will contribute $f(s)$ (if $\sigma>1$ the pole at $z=s-1$ will contribute $-\frac{kf(s-1)}{x}=O(t^{r}x^{-1})$ and if $\sigma=1$ we avoid the pole on the imaginary axis by integrating over half a circle centered at $s-1$ with some small fixed radius). The contribution of a pole $\rho_j\in(0,1)$ of $f$ with multiplicity $1+n_j$ is bounded by $\frac{x^{\rho_j-\sigma}(\log(x))^{n_j}}{|\rho_j-s||\rho_j-s+1|\cdots|\rho_j-s+k|}=O(\frac{x^{1-\sigma}}{t^{k+1}})$ which also bounds the contribution of the simple pole at one. Finally, using \eqref{e:tbound} and the fact that $k>r$ we can bound the remaining integral by $O(t^{r}x^{-\sigma})$ giving that
$$f_{x,k}^*(s)=f(s)+O(\frac{x^{1-\sigma}}{t^{k+1}})+O(t^{r}x^{-\sigma})+O(t^rx^{-1}),$$
where the last error term occurs only when $\sigma>1$.
Consequently, for any $\sigma\geq \sigma_0$ and $x^{\frac{1-\sigma_0}{k+1}}\leq t\leq x^{\frac{\sigma_0}{r}}$ we have $f_{x,k}^*(s)=f(s)+O(1)$ as claimed.
\end{proof}

\begin{proof}[Proof of Proposition \ref{l:boundintf}]
Let $\sigma_1=\frac{4r-1}{4r}$ and fix some $\tfrac{1}{2}\leq \sigma_0<\sigma_1$. For $x=T^{r/\sigma_0}$ and $k>r$ large enough so that $\frac{r(1-\sigma_0)}{(k+1)\sigma_0}<\frac{1}{2r+1}$  let
$f^*(s)=f^*_{x,k}(s)$ as in \eqref{e:f*}.
By Lemma \ref{l:ff*} we have that $f^*(s)=f(s)+O(1)$ for all $\sigma>\sigma_0$ and $T^{\frac{1}{2r+1}}\leq t\leq T$.
We can thus bound for all $\sigma> \sigma_0$
\begin{eqnarray*}\int_1^T|f(\sigma+it)|^2&=& \int_1^{T^{\frac{1}{2r+1}}}|f(\sigma+it)|^2+\int_{T^{\frac{1}{2r+1}}}^T |f(\sigma+it)|^2\\
&\ll& \int_{-T}^T |f^*(\sigma+it)|^2+O(T),
\end{eqnarray*}
where we used \eqref{e:tbound} to bound
$$\int_1^{T^{\frac{1}{2r+1}}}|f(\sigma+it)|^2\ll \int_1^{T^{\frac{1}{2r+1}}}t^{2r}dt\ll T.$$

Next, using that $1\ll\frac{\sin(x)}{x}$ for $|x|\leq\tfrac{1}{2}$ we bound
\begin{eqnarray*}
\int_{-T}^T |f^*(\sigma+it)|^2dt &\ll &  \int_{-T}^T|f^*(\sigma+it)|^2\left(\frac{2T}{t}\sin(\frac{t}{2T})\right)^2dt\\
&\leq & \int_{-\infty}^\infty|f^*(\sigma+it)|^2\left(\frac{2T}{t}\sin(\frac{t}{2T})\right)^2dt\\
&=& 2T\sum_{\lambda_n, \lambda_m<x}\frac{a_n^*a_m^*}{(\lambda_n\lambda_m)^\sigma}\int_{-\infty}^\infty \left(\frac{\sin(t)}{t}\right)^2e^{2iTt\log(\frac{\lambda_m}{\lambda_n})}dt\\
&\ll& T\mathop{\sum_{\lambda_n, \lambda_m<x}}_{|\log(\frac{\lambda_m}{\lambda_n})|\leq \frac{1}{T}}\frac{a_n^*a_m^*}{(\lambda_n\lambda_m)^\sigma}\left(1-T|\log(\frac{\lambda_m}{\lambda_n})|\right),
\end{eqnarray*}
where $a_n^*=a_n(1-\frac{\lambda_n}{x})^k\leq a_n$.
Using the condition that $|\log(\frac{\lambda_m}{\lambda_n})|\leq \frac{1}{T}$ we may restrict the double sum to an outer sum over $\lambda_n\leq x$ and an inner sum over  $\lambda_ne^{-1/T}\leq \lambda_m\leq \lambda_ne^{1/T}$ so
\begin{eqnarray*}
\int_{-T}^T |f^*(\sigma+it)|^2dt&\ll & T\sum_{\lambda_n\leq x}\frac{a_n}{\lambda_n^{\sigma}} \left(\sum_{e^{-1/T}\leq \frac{\lambda_m}{\lambda_n}\leq e^{1/T}}\frac{a_m}{\lambda_m^\sigma}\right).\\
\end{eqnarray*}
Now use  \eqref{e:Af} and summation by parts to bound the inner sum by
\begin{eqnarray*}
\sum_{\lambda_ne^{-1/T}\leq \lambda_m\leq \lambda_ne^{1/T}}\frac{a_m}{\lambda_m^\sigma}\ll \frac{\lambda_n^{1-\sigma}}{T}+\log(\lambda_n)\lambda_n^{\frac{2r-1}{2r}-\sigma}
\end{eqnarray*}
hence
\begin{eqnarray*}
\int_{-T}^T |f^*(\sigma+it)|^2dt &\ll & \sum_{\lambda_n\leq x}\frac{a_n}{\lambda_n^{2\sigma-1}}+T\sum_{\lambda_n\leq x}\frac{a_n\log(\lambda_n)}{\lambda_n^{1+2(\sigma-\sigma_1)}}\\
&\ll&x^{2(1-\sigma)}+T\min\{ \frac{1}{(\sigma-\sigma_1)^2},\log^2(x)\}
\end{eqnarray*}
where we used \eqref{e:Af} and summation by parts to bound the first term by $O(x^{2(1-\sigma)})$, and bounded the sum in the second term by $f'(1+2(\sigma-\sigma_1))\ll \frac{1}{(\sigma-\sigma_1)^2}$ when $\sigma>\sigma_1$ and by $\log(x)\sum_{\lambda_n\leq x}\frac{a_n}{\lambda_n}\ll \log^2(x)$ in general. Recalling that $x=T^{r/\sigma_0}$, for any $\sigma\geq \sigma_1$ we can bound the first term by
$x^{2(1-\sigma)}\leq T^{\frac{2r(1-\sigma_1)}{\sigma_0}}\leq T$ while the second term is bounded by $O(T\min\{\frac{1}{(\sigma-\sigma_1)^2},\log^2(T)\}$.

\end{proof}

\subsection{Zeroes of scattering determinant}
In order to apply the above results on Dirichlet series we express the scattering coefficients as certain Dirichlet series with positive coefficients. Explicitly,
\begin{equation}\label{e:phiLij}
\phi_{ij}(s)=c_{j} \frac{\Gamma(s-\frac{d+1}{2})}{\Gamma(s)}L_{ij}(s),
\end{equation}
where $L_{ij}(s)$ is a Dirichlet series of the form
\begin{equation}\label{e:Lij}
L_{ij}(s)=\sum_{n=0}^\infty \frac{a_{ij}(n)}{\lambda_{ij}(n)^s},
\end{equation}
with $a_{ij}(n)\in \bbN$ and $\lambda_{ij}(n)>0$,
converging in the half plane $\Re(s)>d-1$ with a simple pole at $s=d-1$.
Since sums and products of Dirichlet series is also a Dirichlet series, the scattering determinant has an expression of the form
\begin{equation}\label{e:phitoL}
\vphi(s)=\left(\frac{\Gamma(s-\tfrac{d-1}{2})}{\Gamma(s)}\right)^{\kappa}L(s),
\end{equation}
where $L(s)$ is another a Dirichlet series with real but not necessarily positive coefficients, and the zeroes of $\vphi(s)$ in the half plane $\Re(s)>\tfrac{d-1}{2}$ are the same as the zeroes of $L(s)$.
We can also rewrite it as
\begin{equation}\label{e:Ltof}
L(s)=ab^{d-1-2s}L^*(s)
\end{equation}
for some $a,b>0$ with
\begin{equation}\label{e:Lseries}
L^*(s)=1+\sum_{n=1}^\infty\frac{a_n}{\lambda_n^s}
\end{equation}
with all $a_n\in \bbR$ and $1<\lambda_1<\lambda_2<\ldots$
\begin{rem}
The fact that the scattering coefficients can be expressed as such Dirichlet series is crucial for
understanding the distribution of the poles. It seems that when there is no such expression (e.g., for asymptotically hyperbolic manifolds) the pole distribution is different (see \cite{Muller92}). Even though this result is well known to experts, since we did not find a proof for a general hyperbolic manifold in the literature we include a proof in an appendix.
\end{rem}

We now give a few estimates on $L^*(s)$ and then apply the general results on Dirichlet series to prove Theorem \ref{t:ZeroDist1}.
\begin{prop}\label{p:L*}
The function $L^*(s)$ is holomorphic in $\Re(s)>\frac{d-1}{2}$ except for finitely many poles in $(\tfrac{d-1}{2},d-1]$ and satisfies there
\begin{equation}\label{e:fboundsigma}
|L^*(\sigma+it)|=1+O(e^{-c\sigma}),
\end{equation}
 for some $c>0$ and all $\sigma>>1$ sufficiently large,
 \begin{equation}\label{e:fboundt}
L^*(\sigma+it)=O(|t|^{(d-1)\frac{\kappa}{2}}),
\end{equation}
for $\sigma\geq\frac{d-1}{2}$ and $t>>1$ sufficiently large, and
\begin{equation}\label{e:fcritical}
|L^*(\tfrac{d-1}{2}+it)|=a_\Gamma \left|\frac{\Gamma(\tfrac{d-1}{2}+it)}{\Gamma(it)}\right|^\kappa.
\end{equation}
\end{prop}
\begin{proof}
The first part follows from the holomorphic extension of $\vphi(s)$.
The bound \eqref{e:fboundsigma} follows from the expansion \eqref{e:Lseries} which absolutely converges for $\sigma>d-1$ (one can take the constant $c=\log(\lambda_1)$).

Next, \eqref{e:phitoL} and \eqref{e:Ltof} give
\begin{equation}\label{e:ftophi}
L^*(s)=(ac)^{-1}b^{2s+1-d}\left(\frac{\Gamma(s)}{\Gamma(s-\tfrac{d-1}{2})}\right)^\kappa \vphi(s),
\end{equation}
and since $|\vphi(\tfrac{d-1}{2}+it)|=1$ we get \eqref{e:fcritical}.

Finally to show \eqref{e:fboundt}, for $Y>0$ sufficiently large (but fixed) let
$$E_i^Y(z,s)=\left\lbrace\begin{array}{cc} E_i(z,s) & y_j(z)<Y,\; j=1,\ldots,\kappa\\
E_i(z,s)-\delta_{ij}y_j^s-\phi_{ij}(s)y_j^{d-1-s} & y_j(z)\geq Y
\end{array}\right.,$$
and let $E^Y(z,s)$ denote the column vector with components $E_i^Y(z,s)$. From the Maass-Selberg relations (see \cite[Equation (7.44)]{SelbergHarmonic} and \cite[1.62]{CohenSarnak80}) we get the matrix equation
\begin{eqnarray*}\int_{\calF_\G}E^Y(z,s)E^Y(z,s)^*dz&=&\frac{1}{2\sigma-d-1}(Y^{2\sigma+1-d}I-Y^{d-1-2\sigma}\phi(s)\phi(s)^*)\\
&&+\frac{\phi(s)^*Y^{2it}-\phi(s)Y^{-2it}}{2it},
\end{eqnarray*}
where $s=\sigma+it$ and $I$ is the identity matrix. Since the matrix on the left hand side is positive so is the matrix on the right, implying that
\begin{equation*}
\phi(s)\phi(s)^*\leq Y^{4\sigma-2(d-1)}\left(\sqrt{1+(\tfrac{2\sigma+1-d}{2t})^2}+\frac{2\sigma+1-d}{2t}\right)^2I.
\end{equation*}
Taking the trace gives for $\tfrac{d-1}{2}\leq\sigma\leq d$ the uniform bound
\begin{equation}\label{e:vphibound}
|\phi_{ij}(\sigma+it)|\ll \left(\sqrt{1+(\tfrac{2\sigma+1-d}{2t})^2}+\frac{2\sigma+1-d}{2t}\right).
\end{equation}
In particular, $\vphi(\sigma+it)=O(1)$ for $\sigma\in (\tfrac{d-1}{2},d)$ and $|t|>1$. Combining this with \eqref{e:ftophi} for $\tfrac{d-1}{2}<\sigma\leq d$ gives the bound
$$|L^*(\sigma+it)|\ll \left|\frac{\Gamma(\sigma+it)}{\Gamma(\sigma-\tfrac{d-1}{2}+it)}\right|^\kappa\ll |t|^{\frac{(d-1)\kappa}{2}}.$$
Since $L^*(\sigma+it)=O(1)$ for $\sigma\geq d$ (from the series expression) this concludes the proof of \eqref{e:fboundt}.
\end{proof}

\begin{lem}\label{l:intL*half}
There are constants $B_\G,C_\G$ depending on $\G$ such that
\begin{eqnarray*}\lefteqn{\frac{1}{2\pi}\int_{-T}^T(T-|t|)\log|L^*(\tfrac{d-1}{2}+it)|dt=}\\&&
\tfrac{\kappa(d-1)}{4\pi}T^2\log T+B_\G T^2+C_\G T +O(\log(T)).
\end{eqnarray*}
\end{lem}
\begin{proof}
Let $d-1=2m+\nu$ with $\nu\in\{0,1\}$, and expand
$$\Gamma(\tfrac{d-1}{2}+it)=\prod_{j=1}^m(\tfrac{d-1}{2}-j+it)\Gamma(\tfrac{\nu}{2}+it).$$
Using \eqref{e:fcritical} we evaluate the integral
\begin{eqnarray}\label{e:intlogf}
\frac{1}{2\pi}\int_{-T}^T(T-|t|)\log|L^*(\tfrac{d-1}{2}+it)|dt=\frac{\log(a_\Gamma)}{2\pi}T^2 \\
\nonumber +\frac{\kappa}{\pi}\sum_{j=1}^{m}\int_{0}^T(T-t)\log|it+\tfrac{d-1}{2}-j|dt\\ \nonumber
+\frac{\kappa}{\pi}\int_{0}^T(T-t)\log|\frac{\Gamma(\tfrac{\nu}{2}+it)}{\Gamma(it)}|dt.
\end{eqnarray}
The second term on the right hand side of \eqref{e:intlogf} can be evaluated as
\begin{eqnarray*}\lefteqn{
\frac{\kappa}{\pi}\sum_{j=1}^{m}\int_{0}^T(T-t)\log|it+\tfrac{d-1}{2}-j|dt=} \\
&&=\frac{\kappa m}{\pi}(\frac{T^2}{2}\log T-\frac{3T^2}{4})
+\frac{\kappa}{2\pi}\sum_{j=1}^{m}\int_0^T(T-t)\log(1+\frac{(\tfrac{d-1}{2}-j)^2}{t^2})dt\\
&&=\frac{\kappa m}{2\pi}T^2\log T-\frac{3\kappa m}{4\pi}T^2+\frac{\kappa m(d-m-2)}{8}T+O(\log(T)),
\end{eqnarray*}
and the last term is given by
\begin{eqnarray*}\lefteqn{
\frac{\nu\kappa}{\pi}\int_{0}^T(T-t)\log|\frac{\Gamma(\tfrac{1}{2}+it)}{\Gamma(it)}|dt
=\frac{\nu\kappa}{2\pi}\int_{0}^T(T-t)\log|\frac{t\tanh(\pi t)}{\pi}|dt}\\
&&=\frac{\nu\kappa}{4\pi}T^2\log(T)-(\frac{3\nu\kappa}{8\pi}+\frac{\nu\kappa\log(\pi)}{4\pi})T^2-
\frac{\nu\kappa}{16}T+O(1).
\end{eqnarray*}
Plugging these back in \eqref{e:intlogf} proves the claim with $B_\G=\tfrac{4\log(a_\Gamma)-\kappa(d-1+2\nu\log\pi)}{8\pi}$ and
$C_\G=\tfrac{\kappa(2m(m+\nu-1)-\nu)}{16}$.
\end{proof}

\begin{lem}\label{l:boundintL*}
For $\alpha\geq\alpha_0=d-\frac{5}{4}$ we have
\begin{eqnarray*}
\frac{1}{T}\int_{-T}^T\log|L^*(\alpha+it)|dt\ll \min\{\log\left(\tfrac{1}{\alpha-\alpha_0}\right),\log\log(T)\}
\end{eqnarray*}
\end{lem}
\begin{proof}
First, from \eqref{e:phitoL} and \eqref{e:vphibound} we get that for $|t|<1$ and $\tfrac{d-1}{2}<\alpha<d$
$$\log|L^*(\alpha+it)|\ll 1+\log(1+\frac{1}{t}),$$
so it is enough to bound
$\int_1^T \log|L^*(\alpha+it)|dt$.
Next, since $L(s)=\det(L_{ij}(s))$ we get the bound
$$|L^*(\alpha+it)|\leq \frac{b^{2\alpha+1-d}}{a}\left(\frac{1}{\kappa}\sum_{i,j}|L_{ij}(\alpha+it)|^2\right)^{\kappa/2},$$
and from the inequality between geometric and arithmetic mean we get
\begin{eqnarray}\label{e:boundintL*1}
\frac{1}{T}\int_1^T \log(|L^*(\alpha+it)|)dt\leq (2\alpha+1-d)\log b-\log a \\ \nonumber+\frac{\kappa}{2}\log\left(\frac{1}{T}\int_1^T \frac{1}{\kappa}\sum_{i,j}|L_{ij}(\alpha+it)|^2dt\right).
\end{eqnarray}

For each pair $(i,j)$ consider
$f(s)=L_{ij}(\tfrac{d-1}{2}(s+1))$. This function is still given by a Dirichlet series with positive coefficients, it converges absolutely for $\Re(s)>1$ with a simple pole at $s=1$, it has an analytic continuation to $\Re(s)\geq 0$ except perhaps finitely many poles in $(0,1)$, and it satisfies that $|f(\sigma+it)|=O(t^{\frac{d-1}{2}})$ for $\sigma\geq 0$ and $t>1$ (by \eqref{e:phiLij} and \eqref{e:vphibound}). We can thus apply Proposition \ref{l:boundintf} with $r=\tfrac{d-1}{2}$ to $L_{ij}(\tfrac{d-1}{2}(s+1))$ and get that for $\alpha\geq \alpha_0$,
$$\frac{1}{T}\int_1^T|L_{ij}(\alpha+it)|^2dt\ll \min\{\tfrac{1}{(\alpha-\alpha_0)^2},\log^2(T)\}.$$
Using this bound with \eqref{e:boundintL*1} gives
\begin{eqnarray*}
\frac{1}{T}\int_1^T \log(|L^*(\alpha+it)|)dt
&\ll&\min\{\log(\tfrac{1}{\alpha-\alpha_0}),\log\log(T)\}\\
\end{eqnarray*}
as claimed.
\end{proof}

\begin{lem}\label{l:Arg}
For any $\alpha>\frac{d-1}{2}$ we have
$$\int_\alpha^\infty\mathrm{arg}(L^*(\sigma+iT))d\sigma =O(\log(T)).$$
\end{lem}
\begin{proof}
Since $L^*(\sigma+iT)=1+O(e^{-c\sigma})$ as $\sigma\to\infty$ we have $\mathrm{arg}(L^*(\sigma+iT))\ll e^{-c\sigma}$, so for $\sigma_1>d$ large enough (but fixed), we have
$$\int_\alpha^\infty\mathrm{arg}(L^*(\sigma+iT))d\sigma=\int_\alpha^{\sigma_1}\mathrm{arg}(L^*(\sigma+iT))d\sigma+O(1).$$
For the remaining integral, using the bound $|L^*(\sigma+iT)|=O(T^{\tfrac{d-1}{2}})$ and Titchmarsh's \cite[Lemma 9.2]{Titchmarsh86} we get that $\mathrm{arg}(L^*(\sigma+iT))=O(\log(T))$ and hence the whole integral is bounded by $O(\log(T))$.
\end{proof}
We can now prove Theorem \ref{t:ZeroDist1}, the argument is almost identical to Selberg's \cite{Selberg90} and we include it for the sake of completeness.
\begin{proof}[Proof of Theorem \ref{t:ZeroDist1}]
The zeroes and poles of $\vphi(s)$ in $\Re(s)>\tfrac{d-1}{2}$ are the same as the zeroes and poles of $L^*(s)$ in $\Re(s)>\tfrac{d-1}{2}$. By Proposition \ref{p:L*}, $L^*(s)$ satisfies all the assumptions needed for \cite[Lemma 1,2]{Selberg90}, stating that for any $\alpha\geq\tfrac{d-1}{2}$,
\begin{eqnarray}\label{e:variation}
\mathop{\sum_{|\gamma|\leq T}}_{\beta>\alpha}(T-|\gamma|)(\beta-\alpha)=\frac{1}{2\pi}\int_{-T}^T(T-|t|)\log|L^*(\alpha+it)|dt\\
\nonumber +T\sum_{\sigma_j>\alpha}(\sigma_j-\alpha)+O(\log(T)),
\end{eqnarray}
where the last sum is over the finitely many poles in $(\tfrac{d-1}{2},d]$.
Let $F(\alpha,T)$ denote the left hand side of \eqref{e:variation}, and let
\begin{equation}\label{e:F1}
F_1(\alpha,T)=\mathop{\sum_{|\gamma|\leq T}}_{\beta>\alpha}(\beta-\alpha).
\end{equation}
One easily sees that
\begin{equation}\label{e:diff}
F(\alpha, T)-F(\alpha, T-1)\leq F_1(\alpha,T)\leq F(\alpha,T+1)-F(\alpha,T).
\end{equation}
From  \eqref{e:variation} with $\alpha=\tfrac{d-1}{2}$, together with Lemma \ref{l:intL*half} we get
\begin{eqnarray*}
F(\tfrac{d-1}{2};T)=\tfrac{\kappa(d-1)}{4\pi}T^2\log T+B_\G T^2+(C_\G+\sum_{\sigma_j>\alpha}(\sigma_j-\tfrac{d-1}{2})) T +O(\log(T)),
\end{eqnarray*}
which together with \eqref{e:diff} implies that
$$F_1(\tfrac{d-1}{2};T)=\tfrac{\kappa(d-1)}{2\pi}T\log T+A_\G T+O(\log(T)),$$
with $A_\G=2(C_\G+\sum_{\sigma_j>\alpha}(\sigma_j-\tfrac{d-1}{2}))+B_\G$. This confirms \eqref{e:ZeroDist1}

Next, to show \eqref{e:ZeroDist2}, we use
Littlewood's formula
\begin{eqnarray*}\label{e:Littlewood}
\mathop{\sum_{|\gamma|\leq T}}_{\beta>\alpha}(\beta-\alpha)&=&\frac{1}{2\pi}\int_{-T}^T\log|L^*(\alpha+it)|dt+\frac{1}{\pi}\int_\alpha^\infty\mathrm{arg}(L^*(\sigma+iT))d\sigma\\
\nonumber &&+\sum_{\sigma_j>\alpha}(\sigma_j-\alpha).\\
\end{eqnarray*}
For $\alpha\geq\alpha_0$, we use Lemma  \ref{l:boundintL*} to bound the first integral, Lemma \ref{l:Arg} to bound the second integral, and bound the sum over the poles by $O(1)$ to get that
\begin{eqnarray*}
\mathop{\sum_{|\gamma|\leq T}}_{\beta>\alpha}(\beta-\alpha)\ll T\min\{\log(\tfrac{1}{\alpha-\alpha_0}),\log\log(T)\})
\end{eqnarray*}
confirming \eqref{e:ZeroDist2}.
\end{proof}


\appendix
\section{Scattering determinant as a Dirichlet series}
In this section we verify the formula \eqref{e:phiLij}, expressing the scattering coefficients as Dirichlet series with positive coefficients.
Recall the definition of the Eisenstein series, it's constant terms, and the scattering coefficients given in \eqref{e:EisensteinSph},\eqref{e:ConstSph}, and \eqref{e:ConstSph1}. 
We will show
\begin{prop}\label{p:scatterDirichlet}
The coefficients of the scattering matrix
can be written as
\begin{equation}\label{e:phiLijA}
\phi_{ij}(s)=c_{j} \frac{\Gamma(s-\frac{d+1}{2})}{\Gamma(s)}L_{ij}(s),
\end{equation}
where $L_{ij}(s)$ is a Dirichlet series of the form
\begin{equation}\label{e:LijA}
L_{ij}(s)=\sum_{n=0}^\infty \frac{a_{ij}(n)}{\lambda_{ij}(n)^s},
\end{equation}
with $a_{ij}(n)\in \bbN$ and $\lambda_{ij}(n)>0$,
converging in the half plane $\Re(s)>d-1$ with a simple pole at $s=d-1$.
\end{prop}
We follow the same proof given in \cite{SelbergHarmonic} for hyperbolic surfaces, that is, we explicitly compute the terms $y_i(\gamma.z)^s$ appearing in \eqref{e:EisensteinSph}, and then integrate along the cusps. We will need the following explicit formula for the hyperbolic action.
\begin{lem}\label{l:HypAction}
Let $G\cong\SO_0(d,1)$ denote the group of isometries of hyperbolic space $\Hd$. For any $g\in G$ there is $\lambda=\lambda(g)\geq 0$ such that for any $z=(x,y)\in \Hd$
\begin{equation}\label{e:HypAction}
y(z)=y(g.z)\left\lbrace\begin{array}{cc} \lambda\big(y^2+\norm{x+\eta}^2\big) & \lambda>0\\
 \alpha & \lambda=0\end{array}\right.,
 \end{equation}
for some $\eta\in \bbR^{d-1}$ and $\alpha>0$ (depending only on $g$).
\end{lem}
\begin{proof}
We start with the hyperboloid model for hyperbolic space
\begin{equation}\label{e:HmodelL}
\mathbb{L}^d=\{(\xi_0,\xi_1,\ldots,\xi_d)| \xi_0^2+\ldots \xi_{d-1}^2-\xi_{d}^2=-1,\; \xi_{d}>0\}.
\end{equation}
In this model the group of isometries is just the group of linear  maps sending $\mathbb{L}^d$ onto itself (see  \cite[Section 7]{CannonFloydKenyonParry97} for  more details on the various models of hyperbolic space and the action of the group of isometries on each).
Explicitly, this is the identity component of the group
\begin{equation}\label{e:SOd1}
\SO(d,1)=\{A\in \SL_{d+1}(\bbR)|A^tJA=J=AJA^t\},
\end{equation}
with $J=\mathrm{diag}(1,\dots,1,-1)$, acting linearly on $\mathbb{L}^d\subseteq \bbR^{d+1}$.

In order to see how this action looks like in the upper half space model
we use the isometry $\iota:\mathbb{L}^d\to \Hd$ given by
\begin{equation}\label{e:IsoLH}
\iota(\xi_0,\ldots,\xi_d)=(\frac{2\xi_1}{\xi_0+\xi_d},\ldots, \frac{2\xi_{d-1}}{\xi_0+\xi_d},\frac{2}{\xi_0+\xi_d})=(x,y),
\end{equation}
with inverse map given by
\begin{equation}\label{e:IsoHL}
\iota^{-1}(x,y)=(\tfrac{1-\frac{1}{4}(y^2+\sum_j x_j^2)}{y},\tfrac{x_1}{y},\ldots,\tfrac{x_{d-1}}{y},\tfrac{1+\tfrac{1}{4}(y^2+\sum_j x_j^2)}{y}).
\end{equation}

Now fix some $g\in \mathrm{Isom}(\Hd)$  and let $A\in \SO_0(d,1)$ denote the corresponding linear map acting on $\Ld$. Given $z=(x,y)\in \Hd$ let $\xi=\iota^{-1}(z)\in \Ld$ and $\tilde{\xi}=A\xi$ so that $\iota(\tilde{\xi})=(\tilde{x},\tilde{y})=g.z$. Using \eqref{e:IsoLH} we have that
$$\tilde{y}=\frac{2}{\tilde{\xi}_0+\tilde{\xi}_d},\quad \tilde{x}_j=\tilde{\xi_j}\tilde{y},$$
and since $A=(a_{i,j})$ acts linearly we can write
$$\tilde{\xi}_0+\tilde{\xi}_d=\sum_{j=0}^d \alpha_j\xi_j, \mbox{ with } \alpha_j=a_{0,j}+a_{d,j}.$$
Now use \eqref{e:IsoHL} to rewrite $\xi_j$ back in terms of $(x,y)$ to get
$$\tilde{\xi}_0+\tilde{\xi}_d=\frac{\alpha_0+\alpha_d}{y}+\frac{\alpha_0-\alpha_d}{4y}(y^2+\sum_{j=1}^{d-1}x_j^2)+\sum_{j=1}^{d-1} \frac{\alpha_j x_j}{y}.$$

We first, consider the case where $\alpha_0\neq \alpha_d$ and we define $\lambda=\frac{\alpha_d-\alpha_0}{8}$ and $\eta_j=\frac{2\alpha_j}{\alpha_d-\alpha_0}$. Completing the squares we get
$$
\tilde{\xi}_0+\tilde{\xi}_d=\frac{1}{y}\left[(\alpha_d+\alpha_0)-\frac{1}{\alpha_d-\alpha_0}\sum_{j=1}^{d-1}\alpha_j^2+2\lambda\left(y^2+\sum_{j=1}^{d-1}(x_j+\eta_j)^2\right) \right].
$$
A direct computation using the fact that $A\in \SO(d,1)$ shows that
$$(\alpha_d+\alpha_0)=\frac{1}{\alpha_d-\alpha_0}\sum_{j=1}^{d-1}\alpha_j^2,$$
hence,
$$
\tilde{y}=\frac{2}{\tilde{\xi}_0+\tilde{\xi}_d}=\frac{y}{\lambda\left(y^2+\norm{x+\eta}^2\right)}.
$$
Note that $\lambda$ depends only on $A$ (and hence only on $g$) and not on $(x,y)$, in particular, the fact that $\tilde{y}>0$ implies that $\lambda>0$ as well.

Next, consider the case where $\alpha_0= \alpha_d=\alpha$ so that
$$\tilde{\xi}_0+\tilde{\xi}_d=\frac{2\alpha}{y}+\sum_{j=1}^{d-1} \frac{\alpha_j x_j}{y}.$$
Using again that $\tilde{y}>0$ we must have that $\alpha>0$ and that $\alpha_j=0$ for $j=1,\ldots, d-1$, (otherwise one can always choose $x_j$ to make $\tilde{y}=\frac{2}{\tilde{\xi}_0+\tilde{\xi}_d}<0$). We thus see that $\tilde{\xi}_0+\tilde{\xi}_d=\frac{2\alpha}{y}$, hence, $\tilde{y}=y/\alpha$ concluding the proof.
\end{proof}


With this formula we can compute the scattering matrix and express it in terms of Dirichlet series following the same argument as \cite{SelbergHarmonic}.
\begin{proof}[Proof of Proposition \ref{p:scatterDirichlet}]
We first consider the top left coefficient $\phi_{11}$, given by
 \begin{equation}\label{e:constant}
\frac{1}{\vol(\Gamma_N\bs N)}\int_{\Gamma_N\bs N}E_1(s,n.z)dn=y^s+\phi_{11}(s)y^{d-1-s}.
\end{equation}
We note that $\tau\in \Gamma_N$ acts on $z=(x,y)\in \Hd$ via $(x,y)\mapsto(x+u_\tau,y)$ and that the set,
$L_1\subset \bbR^{d-1}$, of all $u_\tau$'s occurring in this way is a lattice in $\bbR^{d-1}$. Equation \eqref{e:constant} can be written explicitly as
 \begin{equation}\label{e:constant1}\
\frac{1}{v_1}\int_{\calF_1}E(s,(x,y))dx=y^s+\phi_{11}(s)y^{d-1-s},
\end{equation}
where $\calF_1\subseteq \bbR^{d-1}$ is a fundamental domain for $L_1\bs \bbR^{d-1}$ and $v_1=\vol(\calF_1)$.
Using the explicit formula \eqref{e:HypAction} for $y(\gamma.z)$ we can compute this integral directly (for $\Re(s)>d-1$) as
\begin{eqnarray}\label{e:explicit1}\nonumber
\int_{\calF_{1}}E(s,(x,y))dx
 &=& y^s v_1\big(\mathop{\sum_{\gamma\in \Gamma_P\bs \Gamma}}_{\lambda(\gamma)=0}\frac{1}{\alpha(\gamma)}\big)\\
&&+\mathop{\sum_{\gamma\in \Gamma_P\bs \Gamma}}_{\lambda(\gamma)>0}\frac{y^s}{\lambda(\gamma)^s}\int_{\calF_1}\frac{dx}{(y^2+\norm{x+\eta}^2)^s}
\end{eqnarray}
For any $\gamma\in \Gamma_P\bs \Gamma$ with $\lambda(\gamma)>0$ and $\tau\in \Gamma_N$ we have that $\gamma$ and $\gamma\tau$ are distinct classes in $\Gamma_P\bs \Gamma$ with $\lambda(\gamma)=\lambda(\gamma\tau)$. Indeed, a direct computation using \eqref{e:HypAction} shows that
$y(\gamma.z)=y(\gamma\tau.z)$ for all $z$ if and only if $\tau=1$, hence, $\gamma\tau\gamma^{-1}\not\in \Gamma_P$ when $\tau\neq 1$;
To see that $\lambda(\gamma)=\lambda(\gamma\tau)$ use \eqref{e:HypAction} to get the identity
\begin{eqnarray*}
\frac{1}{\lambda(\gamma \tau)}&=&y(z)\cdot y(\gamma\tau.z)=\frac{y^2}{\lambda(\gamma)(y^2+\norm{x+u_\tau+\eta(\gamma)}^2)}\\
&=&\frac{1}{\lambda(\gamma)(1+y^{-2}\norm{x+u_\tau+\eta(\gamma)}^2)}
\end{eqnarray*} and take $y\to\infty$.

We can thus write the second sum in \eqref{e:explicit1} as
\begin{eqnarray*}
 \mathop{\sum_{\gamma\in \Gamma_P\bs \Gamma/\Gamma_N}}_{\lambda(\gamma)>0}\sum_{u\in L_1}\frac{y^s}{\lambda(\gamma)^s}\int_{\calF_1}\frac{dx}{(y^2+\norm{x+u+\eta}^2)^s}\\
 =\mathop{\sum_{\gamma\in \Gamma_P\bs \Gamma/\Gamma_N}}_{\lambda(\gamma)>0}\frac{y^s}{\lambda(\gamma)^s}\int_{\bbR^{d-1}}\frac{dx}{(y^2+\norm{x}^2)^s}\\
 =\mathop{\sum_{\gamma\in \Gamma_P\bs \Gamma/\Gamma_N}}_{\lambda(\gamma)>0}\frac{y^{d-1-s}}{\lambda(\gamma)^s}\int_{\bbR^{d-1}}\frac{dx}{(1+\norm{x}^2)^s}\\
=c(d)\frac{\Gamma(s-\frac{d-1}{2})}{\Gamma(s)}\mathop{\sum_{\gamma\in \Gamma_P\bs \Gamma/\Gamma_N}}_{\lambda(\gamma)>0}\frac{y^{d-1-s}}{\lambda(\gamma)^s}\\
\end{eqnarray*}
Consider the set
$$\Lambda_{11}=\{\lambda(\gamma)|\gamma\in \Gamma_P\bs\Gamma/\Gamma_N\}\cap (0,\infty),$$
from the discreteness of $\G$ we get that $\Lambda_{11}\subset (0,\infty)$ is discrete and we can order it
\begin{eqnarray*}
\Lambda_{11}
=\{\lambda_{11}(0)<\lambda_{11}(1)<\lambda_{11}(2)<\cdots\},\end{eqnarray*}
and let
$$a_{11}(n)=\#\{\gamma\in \Gamma_P\bs\Gamma/\Gamma_N|\lambda(\gamma)=\lambda_{11}(n)\}.$$
With these notation we get that
\begin{equation}\label{e:explicit2}
\frac{1}{v_1}\int_{\calF_1}E_1(s,z)dx= y^s \big(\mathop{\sum_{\gamma\in \Gamma_\infty\bs \Gamma}}_{\lambda(\gamma)=0}\frac{1}{\alpha(\gamma)}\big)+
\frac{c(d)}{v_1}\frac{\Gamma(s-\frac{d-1}{2})}{\Gamma(s)}\sum_{n=0}^\infty \frac{a_{11}(n)}{\lambda_{11}(n)^s}
\end{equation}
Comparing \eqref{e:constant1} with \eqref{e:explicit2} we see that
$$\mathop{\sum_{\gamma\in \Gamma_\infty\bs \Gamma}}_{\lambda(\gamma)=0}\frac{1}{\alpha(\gamma)}=1,$$
and that
$$\phi_{11}(s)=c_1\frac{\Gamma(s-\frac{d-1}{2})}{\Gamma(s)}L_{11}(s),$$
with $c_1=\frac{c(d)}{v_1}$.
Finally, since the Eisenstein series $E_1(s,z)$ absolutely converges for $\Re(s)>d-1$ and has a simple pole at $s=d-1$, the series $L_{11}(s)$ also absolutely converges in this region and has a simple pole at $s=d-1$.

The formula for the other coefficients $\phi_{ij}(s)$ follows from the same arguments where we denote by
$\lambda_{ij}(\gamma)=\lambda(k_i^{-1}\gamma k_j)$ and let
\begin{eqnarray*}\Lambda_{ij}&=&\{\lambda_{ij}(\gamma)>0|\gamma\in \Gamma_{P_i}\bs\Gamma/\Gamma_{N_j}\}\\
&=&\{\lambda_{ij}(0)<\lambda_{ij}(1)<\lambda_{ij}(2)<\cdots\},
\end{eqnarray*}
and $a_{ij}(n)=\#\{\gamma\in \Gamma_{P_i}\bs\Gamma/\Gamma_{N_j}|\lambda_{ij}(\gamma)=\lambda_{ij}(n)\}$.
Note that the fact that the cusps are distinct implies that when $i\neq j$ we have that $\lambda_{ij}(\gamma)>0$ for all $\gamma\in \Gamma$ and the first sum in
\eqref{e:explicit2} vanishes.
\end{proof}


\def\cprime{$'$} \def\cprime{$'$}


\end{document}